\documentclass[a4paper]{amsart}

\usepackage[english]{babel}
\usepackage[latin1]{inputenc}
\usepackage{enumerate} 
\usepackage{latexsym} 
\usepackage{amssymb} 
\usepackage{amsmath} 
\usepackage[all,arc,curve,color,frame]{xy}
\usepackage{multirow} 

\newtheorem{thm}{Theorem}[section]
\newtheorem{lem}[thm]{Lemma}

\newtheorem{defi}[thm]{Definition}
\newtheorem{prop}[thm]{Proposition}
\newtheorem{exam}[thm]{Example}

\theoremstyle{remark}
\newtheorem{rema}[thm]{Remark}
\newtheorem{nota}[thm]{Notation}

\newcommand{\N}{\mathbb{N}}
\newcommand{\E}{\mathcal{E}}

\DeclareMathOperator{\rank}{rank}
\DeclareMathOperator{\coker}{coker}
\DeclareMathOperator{\sgn}{sgn}

\newcommand{\Ceq}{\sim_{C^*}}
\newcommand{\Keq}{\sim_{K}}
\newcommand{\Meq}{\sim_{M}}
\newcommand{\MCeq}{\sim_{M'}}

\begin{document}

\title{Geometric classification of simple graph algebras}
\author{Adam P. W. S{\o}rensen}
\address{Department of Mathematical Sciences\\
University of Copenhagen\\
Universitetsparken 5, DK-2100, Copenhagen \O, Denmark}
\email{apws@math.ku.dk}



\date{\today}

\thanks{
This research was supported by the Danish National Research Foundation through the Centre for Symmetry and Deformation.}

\begin{abstract}
Inspired by Franks' classification of irreducible shifts of finite type we provide a short list of allowed moves on graphs that preserves the stable isomorphism class of the associated $C^*$-algebras.
We show that if two graphs have stably isomorphic and simple unital algebras then we can use these moves to transform one into the other. 
\end{abstract}

\maketitle

\section{Introduction}

In \cite{ParrySullivan} Parry and Sullivan answered Bowens question, ``what is the equivalence relation on non-negative integral matrices induced by flow equivalence?'' 
They did this by showing that flow equivalence is generated by two moves on matrices, \emph{strong shift equivalence} and what is now known as the \emph{Parry-Sullivan move}.
Using this result, Franks classified irreducible subshifts of finite type up to flow equivalence \cite{franks_flowequivalence}.

Cuntz and Krieger noticed a connection between $C^*$-algebras and subshifts of finite type, and in \cite{cuntzKrieger_markovChains} they associated to any $0$-$1$ valued square matrix $A$ with no zero rows or columns a $C^*$-algebra $\mathcal{O}_A$.
These algebras are now known as Cuntz-Krieger algebras.
They reflect many of the properties of the matrix, for instance if $A$ is an irreducible non-permutation matrix then $\mathcal{O}_A$ is simple, and they are an invariant for flow equivalence in the following sense: if $B$ and $C$ are flow equivalent irreducible matrices then $\mathcal{O}_B$ is stably isomorphic to $\mathcal{O}_C$.
The latter observation raised the question whether flow equivalence of irreducible matrices is equivalent to stable isomorphism of the associated $C^*$-algebras.
Building on an idea by Cuntz, R{\o}rdam answered this question in the negative in  \cite{rordam_classCuntzKriger}.
R{\o}rdam showed that the relation on (irreducible non-permutation $0$-$1$ valued square) matrices $A$ and $B$ given by $A \sim B$ if and only if $\mathcal{O}_A$ is stably isomorphic to $\mathcal{O}_B$, is generated not only by strong shift equivalence and the Parry-Sullivan move, but also the relation 
\[
 A \sim A\_ =	\left( \begin{array}{cc|cccc} 
 								1 & 1 & 1 & 0 & 0 & \cdots \\
 								1 & 1 & 0 & 0 & 0 & \cdots \\
 								\hline 
 								1 & 0 & \multicolumn{4}{c}{}\\
 								0 & 0 & \multicolumn{4}{c}{A} \\ 
 								\vdots & \vdots & \multicolumn{4}{c}{}
 						  \end{array} \right).
\]
The operation that takes $A$ to $A\_$ is called a \emph{Cuntz splice}. 
It is known, due to the work of Franks, that two irreducible matrices $A, B$ are flow equivalent if and only if:
\begin{enumerate}
	\item $\coker(I - A) \cong \coker(I - B)$, and
	\item $\det(I - A) = \det(I - B)$.
\end{enumerate}
In fact, given $\coker(I - A)$ we can easily construct the Smith normal form, $S$ say, of $I - A$, and since $|\det(I - A)| = \det(S)$ we can reduce $2$ to $\sgn \det(I - A) = \sgn \det(I - B)$.
The importance of the relation $A \sim A\_$ becomes clear, once we notice that $\coker(I - A) \cong \coker(I - A\_)$, but $\det(I - A) = - \det(I - A\_)$.
Hence, R{\o}rdam's result states that the determinant of $I - A$ is not an obstruction for stable isomorphism of simple Cuntz-Krieger algebras, so they are classified by their $K$-theory, as $K_0(\mathcal{O}_A) \cong \coker(I - A)$.

In the reducible matrix/non-simple $C^*$-algebra case, Huang has classified reducible subshifts of finite type using the so-called $K$-web, \cite{huang_kweb}.
Restorff used results of Boyle and Huang (\cite{mbdh_poseteq, boyle_floweq}) to prove that Cuntz-Kriger algebras with finitely many ideals are classified by ideal related $K$-theory, \cite{restorff_CK}.

Graph algebras are a generalization of Cuntz-Krieger algebras.
For some of the first steps towards this generalization see \cite{ew_japon} and \cite{kprr_groupoids}. 
The simple graph algebras are either purely infinite or AF, and so are classifiable by $K$-theory. 
Certain non-simple graph algebras have also been classified using $K$-theoretic methods, for instance the case of precisely one ideal is handled in \cite{eilersTomforde}, and some linear ideal latices are considered in \cite{err_malaysian}. 
These results are proved using heavy classification machinery, and using that the class of graph algebras behaves nicely with respect to classification. 

In this paper we take an approach that is less $K$-theoretic and much closer to how the subshifts of finite type were classified, as we study moves on graphs that preserve the stable isomorphism class of the involved graph algebras.
Our focus will be on graphs with finitely many vertices, i.e. on unital graph algebras, since this makes it seem plausible that a finite number of moves can transform related graphs into one another. 
One could hope to find a (short) list of moves such that two graphs have stably isomorphic $C^*$-algebras if and only if we can transform one into the other using these moves, similar to the situation for flow equivalence.
Moves on special types of graphs have been studied in \cite{efw_cuntzkriegerclass, ddns_eqgraphs} where the graphs must have either one or zero edges between any two vertices, and in \cite{ers_amp} where the graphs must have either infinitely many or zero edges between any two vertices.
Leavitt path algebras, the the algebraic cousins of the analytic graph $C^*$-algebras, have also been studied using flow equivalence techniques and moves. 
In \cite{alps_flow} the isomorphism question for purely infinite simple algebras is studied using Franks' invariant, and in \cite{aalp_gauge} an algebraic version of the gauge-invariant uniqueness theorem is used to provide moves on graph that preserves isomorphism of Leavitt path algebras. 
(The author is grateful to Enrique Pardo for calling his attention to this work.)
 
Inspired by the moves used in \cite{ers_amp} to control infinite emitters, we will study arbitrary graphs and find a short list of moves that suffices if the involved graph algebras are simple and unital (for a precise statement see Theorem~\ref{mainThm}).
We stress that the moves (with one exception) preserve the stable isomorphism class of any graph algebra, but that we only prove they generate the equivalence relation given by stable isomorphism of the algebras in the class of unital simple graph algebras. 
As we already noted, simple unital graph algebras are classified by $K$-theory, so the chronology here is backwards when compared to subshifts of finite type. 
Hopefully classification of graph algebras by moves will one day catch up and maybe even overtake classification by $K$-theoretic means. 

The basic moves we use to manipulate graphs are:

\begin{enumerate}
	\item[{\tt (S)}] Remove a source, if it is a regular vertex,
	\item[{\tt (I)}] In-split the graph (as described in Theorem~\ref{inSplit}),
	\item[{\tt (O)}] Out-split the graph (as described in Theorem~\ref{outSplit}),
	\item[{\tt (R)}] Reduction (as described in Proposition~\ref{moveR}). 
\end{enumerate}

In- and out-splitings for graph algebras originate from \cite{batesFlow}.
The moves {\tt (S)}, {\tt (I)}, {\tt (O)}, and {\tt (R)} alone are not enough: we also need a version of the Cuntz splice (see Definition~\ref{cuntzSplice}) to fix the determinant of the adjacency matrices.
Somewhat surprisingly the Cuntz splice is only needed when we study graphs without any infinite emitters. 

\section{Notation}

We use this short section to fix notation and give standard definitions. 
First and foremost: in this paper we follow the convention for graph algebras used in (for instance) \cite{batesFlow}, but not in (for instance) \cite{raeburn_book}.

We now define a few graph concepts, and give the definition of a graph algebra. 

\begin{defi}
A graph $G$ is a 4-tuple $G = (G^0, G^1, r,s)$ consisting of a set of verices, $G^0$, a set of edges, $G^1$, and two maps $r,s \colon G^1 \to G^0$ specifying the range and source of any edge. 
\end{defi}

\begin{defi}
A path in a graph is a finite sequence of edges $e_1 e_2 \cdots e_n$ such that $r(e_i) = s(e_{i+1})$ for $i=1,2,\ldots, n-1$.
We extend the range and source maps to paths by putting $s(e_1 e_2 \cdots e_n) = s(e_1)$ and $r(e_1 e_2 \cdots e_n) = r(e_n)$. 

A loop is a path $\alpha = \alpha_1 \alpha_2 \cdots \alpha_n$ such that $s(\alpha) = r(\alpha)$, and a simple loop is a loop such that $r(\alpha_i) \neq r(\alpha)$ for any $i \neq n$. 

We say that a vertex $v$ supports a loop or is the base of a loop, if there is some loop $\alpha$ such that $s(\alpha) = v = r(\alpha)$.
\end{defi}

\begin{nota}
Let $G$ be a graph, and let $u,v$ be vertices. 
We write $u \geq v$ if there is a path from $u$ to $v$ or $u = v$. 

If $H$ is a set of vertices we write $H \geq v$ if there is a vertex $w \in H$ such that $w \geq v$, we write $v \geq H$ if there is some vertex $w \in H$ such that $v \geq w$. 
\end{nota}

\begin{defi}
Let $G$ be a graph.  
A vertex $v \in G^0$ is called a source if does not receive any edges, i.e. $r^{-1}(v) = \emptyset$. 
We call a vertex $u \in G^0$ a sink if it does not emit any edges, that is $s^{-1}(u) = \emptyset$. 
If a vertex $w \in G^0$ emits infinitely many edges, meaning $|s^{-1}(w)| = \infty$, we say that $w$ is an infinite emitter. 

A vertex is called singular if it is either an infinite emitter or a sink.
A regular vertex is a vertex that is not singular. 
\end{defi}

\begin{defi}
Let $G = (G^0,G^1,r,s)$ be a graph. The graph $C^*$-algebra of $G$, denoted by $C^*(G)$, is the universal $C^*$-algebra generated by a set of mutually orthogonal projections $\{ p_v \mid v \in G^0 \}$ and a set $\{ s_e \mid e \in G^1 \}$ of partial isometries satisfying the following conditions:
\begin{itemize}
	\item $s_e^* s_f = 0$ if $e,f \in G^1$ and $e \neq f$,
	\item $s_e^* s_e = p_{r_{G}(e)}$ for all $e \in G^1$,
	\item $s_e s_e^* \leq p_{s_{G}(e)}$ for all $e \in G^1$, and,
	\item $p_v = \sum_{e \in s_{G}^{-1}(v)} s_e s_e^*$ for all $v \in G^0$ with $0 < |s_{G}^{-1}(v)| < \infty$.
\end{itemize}
\end{defi}

We would like to discuss matrices as well as graphs so we set up some notation and give a brief description of the Smith normal form. 

\begin{nota}
Throughout this paper we will only consider matrices with entries in $\N \cup \{ 0, \infty \}$.
\end{nota}

\begin{nota}
Given a finite set $X$ and an $X \times X$ matrix $A$, let $G_A$ be the graph with $G_A^0 = X$ and $|s^{-1}(x) \cap r^{-1}(y)| = A(x,y)$ for all $x,y \in X$. 
Given a graph $G$ we let $A_G$ be the $G^0 \times G^0$ matrix with $A(u,v) = |s^{-1}(u) \cap r^{-1}(v)|$. 
\end{nota}

Every integer matrix $A$ without zero entries can be diagonalized, using row and column operations, in such a way that the diagonal entries $d_1, d_2, \ldots, d_n$ (listed from top to bottom) satisfy that $d_i$ divides $d_{i+1}$ and that $d_1$ is the greatest common divisor of all the entries in $A$. 
The diagonal matrix is called the Smith normal form of $A$, and the $d_i$ are called elementary divisors of $A$ or invariant factors of $A$. 
The term elementary divisor is used in \cite{franks_flowequivalence}, so we will also use it here. 
There is an algorithm for computing the Smith form $S$ of a matrix $A$ and since $A$ and $S$ have the same cokernel this makes for convenient computations of cokernels of matrices and therefore of $K$-groups of graph algebras. 
This also illustrates why the elementary divisors played an important role in \cite{franks_flowequivalence} and why they will feature prominently here, see Proposition \ref{standardForm}.
For details about the Smith form see \cite[Section 9.4]{rotman_algebra}

\section{The moves}

In this section we will discuss the basic moves we will use to manipulate graphs. 
Our first two moves are special cases of the slightly complicated move described in \cite[Theorem 3.1]{CrispGow}. 
One does not need the full power of this theorem to prove that these moves preserve Morita equivalence, as this is easily proved by directly writing down explicit isomorphisms, for instance by following the proof given by Crisp and Gow and noticing that there are many simplifications in  the two special cases. 

\begin{prop}[Move {\tt (S)}] \label{moveS}
Let $G = (G^0,G^1,r_G,s_G)$ be a graph, and let $u \in G^0$ be a regular vertex that is a source. 
Define a graph $E = (E^0,E^1, r_E, s_E)$ by $E^0 = G^0 \setminus \{ u \}$, $E^1 = G^1 \setminus s^{-1}(v)$, $r_E = r_G \vert_{E^{0}}$, and $s_E = s_G \vert_{E^0}$.
Then $C^*(E)$ is (isomorphic to) a full corner of $C^*(G)$.
\end{prop}

\begin{prop}[Move {\tt (R)}] \label{moveR}
Let $G = (G^0, G^1, r_G, s_G)$ be a graph, and let $u \in G^0$ be a regular vertex such that $s_G^{-1}(u)$ and $s_G(r_G^{-1}(u))$ are one point sets. 
Let $v$ be the only vertex that emits to $u$ and let $f$ be the only edge $u$ emits. 
Define a graph $E = (E^0, E^1, r_E, s_E)$ by $E^0 = G^0 \setminus \{ u \}$,
\[
	E^1 = (G^1 \setminus (r_G^{-1}(u)) \cup \{ f \}) \cup \{ [ef] \mid e \in r_G^{-1}(u) \},
\]
and range and source maps that extend those of $G$ and satisfy $r_E([ef]) = r_G(f)$ and $s_E([ef]) = s_G(e) = v$.
If $r_G(f) \neq u$ then $C^*(E)$ is (isomorphic to) a full corner of $C^*(G)$.
\end{prop}

The following figure illustrates an application of move {\tt (R)}. 
First we use the move to remove the vertex $\star$, and then we apply it to remove the vertex $\circ$. 
Here we use $\rightsquigarrow$ to denote an application of move {\tt (R)}. 

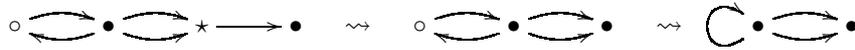
\[
	\xymatrix{ 
		\circ \ar@/^0.5em/[r] & \bullet \ar@/^0.5em/[l] \ar@/^0.5em/[r] \ar@/_0.5em/[r] & \star \ar[r] & \bullet
	} 
	\quad
	\rightsquigarrow
	\quad 
	\xymatrix{ 
		\circ \ar@/^0.5em/[r] & \bullet \ar@/^0.5em/[l] \ar@/^0.5em/[r] \ar@/_0.5em/[r] & \bullet
	} 
	\quad
	\rightsquigarrow
	\quad
	\quad
	\xymatrix{ 
		\bullet \ar@(dl,ul) \ar@/^0.5em/[r] \ar@/_0.5em/[r] & \bullet
	} 
\]
\linebreak 

Aside from these moves, we will also use the (proper) in- and out-splittings at one vertex from \cite{batesFlow}, which we record here for the convenience of the reader.
The following are special cases of \cite[Theorem 3.5 and Corollary 5.4]{batesFlow}.

\begin{thm}[Move {\tt (O)}] \label{outSplit}
Let $G = (G^0, G^1, r,s)$ be a graph, and let $v \in G^0$.
Suppose that $v$ is not a sink. 
Partition $s^{-1}(v)$ into a finite number, $n$ say, of sets $\{ \mathcal{E}_1, \mathcal{E}_2, \ldots, \mathcal{E}_n \}$.
Define a graph $G_{os}$ by:
\begin{align*}
 	G_{os}^0 & = (G^0 \setminus \{ v \}) \cup \{ v^1, v^2, \ldots, v^n \}, \\
	G_{os}^1 & = \left( G^1 \setminus r^{-1}(v) \right) \cup \{ e^1, e^2, \ldots, e^n \mid e \in G^1, r(e) = v \}.
\end{align*}
For $e \notin r^{-1}(v)$ we let $r_{os}(e) = r(e)$, for $e \in r^{-1}(v)$ we let $r_{os}(e^i) = v^i$, $i = 1, 2, \ldots,n$.
For $e \notin s^{-1}(v)$ we let $s_{os}(v) = s(e)$, for $e \in s^{-1}(v) \setminus r^{-1}(v)$ we let $s_{os}(e) = v^i$ if $e \in \mathcal{E}_i$, and for $e \in s^{-1}(v) \cap r^{-1}(v)$ we let $s_{os}(e^j) = v^i$ if $e \in \mathcal{E}$, for $i,j= 1,2, \ldots, n$.

If at most one of the $\mathcal{E}_i$ is infinite then $C^*(G) \cong C^*(G_{os})$.
\end{thm} 

\begin{rema}
If we apply move {\tt (O)} to $G$ thus yielding a graph $E$, we say that $E$ is an out-split of $G$. 
We will also say that $G$ as an out-amalgamation of $E$, and refer to using move {\tt (O)} ``backwards'' as out-amalgamating $E$ into $G$. 
\end{rema}

\begin{thm}[Move {\tt (I)}] \label{inSplit}
Let $G = (G^0, G^1, r,s)$ be a graph, and let $v \in G^0$.
Suppose that $v$ is not a source. 
Partition $r^{-1}(v)$ into a finite number, $n$ say, of sets $\{ \mathcal{E}_1, \mathcal{E}_2, \ldots, \mathcal{E}_n \}$.
Define a graph $G_{is}$ by:
\begin{align*}
 	G_{is}^0 & = (G^0 \setminus \{ v \}) \cup \{ v^1, v^2, \ldots, v^n \}, \\
	G_{is}^1 & = \left( G^1 \setminus s^{-1}(v) \right) \cup \{ e^1, e^2, \ldots, e^n \mid e \in G^1, s(e) = v \}.
\end{align*}
For $e \notin r_{G}^{-1}(v)$ we let $r_{is}(e) = r(e)$, for $e \in r^{-1}(v) \setminus s^{-1}(v)$ we let $r_{is}(e) = v^i$ if $e \in \mathcal{E}_i$, and if $ e \in r^{-1}(v) \cap s^{-1}(v)$ then $r_{is}(e^j) = v^i$ for $e \in \mathcal{E}_i$,  for $i,j=1,2,\ldots, n$.
For $e \notin s^{-1}(v)$ we let $s_{is}(v) = s(e)$, for $e \in s^{-1}(v)$ we let $s_{is}(e^i) = v^i$ for $i= 1,2, \ldots, n$.

If $v$ is a regular vertex then $C^*(G)$ is stably isomorphic to $C^*(G_{is})$.
\end{thm} 

We will also need the Cuntz splice. 

\begin{defi}[{\tt (C)}] \label{cuntzSplice}
Let $G = (G^0, G^1, r_G, s_G)$ be a graph and let $v \in G^0$ be a regular vertex that supports at least two simple loops (recall that a loop $\alpha$ is simple if $r(\alpha_i) \neq r(\alpha)$ for all $i < |\alpha|$). 
Define a graph $E = (E^0, E^1, r_E, s_E)$ by $E^0 = G^0 \cup \{ u_1, u_2 \}$, $E^1 = G^1 \cup \{ e_1, e_2, f_1, f_2, h_1, h_2\}$, $r_E$ and $s_E$ extend $r_G$ and $s_G$ respectively, and satisfy 
\[
	s_E(e_1) = v, \quad s_E(e_2) = u_1, \quad s_E(f_i) = u_1, \quad s_E(h_i) = u_2,
\]
and 
\[
	r_E(e_1) = u_1, \quad r_E(e_2) = v, \quad r_E(f_i) = u_i, \quad r_E(h_i) = u_i.
\]
We say that $E$ arises by applying move {\tt (C)} to $G$ at $v$. 
\end{defi}  

The following is an illustration of move {\tt (C)} at vertex $\star$ 

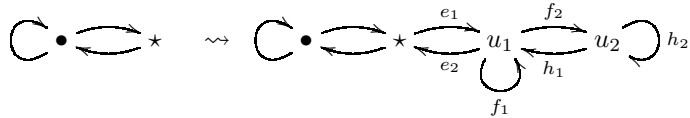
\[
	\xymatrix{ 
		\bullet \ar@(dl,ul) \ar@/^0.5em/[r] & \star \ar@/^0.5em/[l]
	} 
	\quad
	\rightsquigarrow
	\quad
	\quad
	\xymatrix{ 
		\bullet \ar@(dl,ul) \ar@/^0.5em/[r] & \star \ar@/^0.5em/[l] \ar@/^0.5em/[r]^{e_1} & u_1 \ar@/^0.5em/[l]^{e_2} \ar@/^0.5em/[r]^{f_2} \ar@(dl,dr)_{f_1} & u_2 \ar@/^0.5em/[l]^{h_1} \ar@(ur,dr)^{h_2}
	}
\]

\begin{rema} \label{cuntzPreserve}
The Cuntz splice differs from the other moves in that we cannot show that it preserves stable isomorphism of the graph algebras. 
We can however show that it does not change the $K$-theory of the algebra.
This is easily seen once we recall that the $K$-groups are computed as the kernel and cokernel of $I - A_G$, see \cite{ddmt_kthygraph}. 
Since if $E$ arises from $G$ by an application of move {\tt (C)}, then $I - A_E$ can be transformed into 
\[
\left( \begin{array}{cc|cccc} 
 								1 & 0 & 0 & 0 & 0 & \cdots \\
 								0 & 1 & 0 & 0 & 0 & \cdots \\
 								\hline 
 								0 & 0 & \multicolumn{4}{c}{}\\
 								0 & 0 & \multicolumn{4}{c}{I - A_G} \\ 
 								\vdots & \vdots & \multicolumn{4}{c}{}
 						  \end{array} \right).
\]
using row an column operations, so neither the kernel nor the cokernel changes. 
Similarly one can see that neither the ideal structure nor the order of $K_0$ is changed, the latter uses that the Cuntz splice is done at a vertex that supports two simple loops. 
\end{rema}

\section{Some equivalence relations}

We will define four equivalence relations on graphs. 
First our two most important ones.

\begin{defi}
We define $\Meq$ to be the smallest equivalence relation on graphs with finitely many vertices such that $G \Meq E$ if $G$ differs, up to isomorphism of graphs, from $E$ by an application of one of the moves {\tt (S)}, {\tt (I)}, {\tt (O)}, or {\tt (R)}. 

If $G \Meq E$, we say that $G$ is move-equivalent to $E$.
\end{defi}

\begin{defi}
Given two graph $G$ and $E$, we say that $G$ and $E$ are $C^*$-equivalent, written $G \Ceq E$, if $C^*(G) \otimes \mathcal{K} \cong C^*(E) \otimes \mathcal{K}$. 
\end{defi}

We have two more equivalence relations on graphs.

\begin{defi}
Given two graph $G$ and $E$, we say that $G$ and $E$ are K-equivalent, written $G \Keq E$, if $FK(C^*(G)) \cong FK(C^*(E))$, where $FK(-)$ denotes the filtered $K$-theory. See \cite{err_malaysian} for a definition of $FK$.
\end{defi}

\begin{rema}
Suppose that $C^*(G)$ is simple (this is the case we are mainly interested in).
Then
\[
	FK(C^*(G)) = (K_0(C^*(G)), K_0^+(C^*(G)), K_1(C^*(G))).
\] 
So $G \Keq E$ means that $K_0(C^*(G))$ and $K_0(C^*(E))$ are isomorphic as ordered groups, and that $K_1(C^*(G))$ and $K_1(C^*(E))$ are isomorphic as groups. 
\end{rema}

\begin{defi}
We define $\MCeq$ to be the smallest equivalence relation on graphs such that 
\begin{enumerate}
	\item $G \MCeq E$ if $G \Meq E$, and 
	\item $G \MCeq E$ if E arises by using move {\tt (C)} on $G$.  
\end{enumerate}
\end{defi}

\begin{rema}
We have 
\[
	G \Meq E \implies G \Ceq E \implies G \Keq E,
\] 
and, by Remark \ref{cuntzPreserve},
\[
	G \MCeq E \implies G \Keq E. 
\]
If $C^*(G)$ and $C^*(E)$ are purely infinite simple then by the Kirchberg-Phillips theorem (for instance \cite[Theorem 4.2.4]{phillips})
\[
	G \Keq E \implies G \Ceq E.
\]
\end{rema}

We also define the corresponding equivalence relations on matrices. 

\begin{defi}
Given two square matrices $A$ and $B$, we say that $A$ and $B$ are $C^*$-equivalent, written $A \Ceq B$, if $G_A \Ceq G_B$.
Likewise $A \Meq B$ if $G_A \Meq G_B$, $A \Keq B$ if $G_A \Keq G_B$, and $A \MCeq B$ if $G_A \MCeq G_B$. 
\end{defi}

Using the notation introduced in this section we can state our main result. 

\begin{thm} \label{mainThm}
Let $G,E$ be graphs with simple unital algebras. 
If $G$ has at least one singularity then 
\[
	G \Keq E \iff G \Meq E \iff G \Ceq E.
\]
If $G$ has no singularities then 
\[
	G \Keq E \iff G \MCeq E \iff G \Ceq E.
\]
\end{thm}

We will prove the theorem in section~\ref{sec:classification}.

\begin{rema}
The case where the algebras are purely infinite and the graphs have no infinite emitters is essentially Franks' result \cite[Theorem 3.3]{franks_flowequivalence} combined with R{\o}rdam's result \cite[Theorem 6.5]{rordam_classCuntzKriger}. 
\end{rema}

\begin{rema}
Suppose that $G$ and $E$ are graphs with simple unital algebras, that $G$ has atleast one singularity, and that $G \Keq E$.
By Theorem \ref{mainThm} we then have $G \Meq E$.
Since the maps showing that the moves {\tt (S)}, {\tt (R)}, {\tt (I)} and {\tt (O)} preserve stable isomorphism can be described explicitly, one could chase through the moves used to show $G \Meq E$ and thereby obtain a chain of concrete isomorphism of full corners showing that $C^*(G)$ is stably isomorphic to $C^*(E)$.  
However, this is probably impractical in most cases. 
\end{rema}

\section{Derived moves}

We will now describe two very usable moves that are in $\Meq$ but not on our list, one involving infinite emitters and one involving regular vertices.
We chose not to include them in the initial list of allowed moves to get an as short and simple list as possible. 
Though we will use these moves to study graphs with simple unital algebras, the use of the moves require no such restriction. 
First we deal with regular vertices. 

\subsection{Collapse}

This move brings us a lot closer to the full power of \cite[Theorem 3.1]{CrispGow}. 
We will prove it in two steps. 
First we prove that we can handle the case where the vertex we collapse receives edges from more than one vertex. 

\begin{lem} \label{easyCollapse}
Let $G = (G^0, G^1, r_G, s_G)$ be a graph with finitely many vertices, and let $v \in G^0$ be a regular vertex that is not a source and which emits precisely one edge, $f_0$ say. 
Define a graph $E = (E^0, E^1, r_E, s_E)$ by $E^0 = G^0 \setminus \{ v \}$,
\[
	E^1 = \left( G^1 \setminus (r^{-1}(v) \cup s^{-1}(v)) \right) \bigcup \left( \{ [ef_0] \mid e \in r^{-1}(v) \} \right),
\]
the range and source maps extend those of $G$ and satisfy $r_E([ef_0]) = r_G(f_0)$ and $s_E([ef]) = s_G(e)$. 
If $f_0$ is not a loop of length one then $G \Meq E$.
\end{lem}
\begin{proof}
If $v$ receives from only one vertex then this is move {\tt (R)}. 
So let us assume that $v$ receives from the vertices $\{ u_1, u_2, \ldots, u_n \}$, where $2 \leq n < \infty$. 
Define $\E_i = r^{-1}(v) \cap s^{-1}(u_i)$ for $i = 1,2,\ldots, n$.
Since $v$ is regular we can use move {\tt (I)} at $v$ according to the partition $\{\E_i\}$. 
This will yield a graph $G_{in}$ with $v$ replaced by $n$ vertices $v_1, v_2, \ldots, v_n$ each receiving from one of the vertices $v$ received from and each emitting one edge to $r(f_0)$.
We can now use move {\tt (R)} to collapse the $v_i$, this will yield the graph $E$. 
Hence $G \Meq E$.
\end{proof}

We now deal with the general case. 

\begin{thm}[Collapse] \label{collapse}
Let $G = (G^0, G^1, r_G, s_G)$ be a graph with finitely many vertices, and let $v \in G^0$ be a regular vertex which does not support a loop of length one. 
Define a graph $E = (E^0, E^1, r_E, s_E)$ by $E^0 = G^0 \setminus \{ v \}$,
\[
	E^1 = \left( G^1 \setminus (r^{-1}(v) \cup s^{-1}(v)) \right) \bigcup \left( \{ [ef] \mid e \in r^{-1}(v), f \in s^{-1}(v) \} \right),
\]
the range and source maps extend those of $G$, and satisfy $r_E([ef]) = r_G(f)$ and $s_E([ef]) = s_G(e)$. 
We have $G \Meq E$.
\end{thm}
\begin{proof}
Let $\{ e_1, e_2, \ldots, e_n \}$ be the edges with source $v$.
If $n = 1$ then we appeal to Lemma~\ref{easyCollapse}, otherwise we define $\E_i = \{ e_i \}$, for $i = 1,2,\ldots, n$. 
We can use move {\tt (O)} at $v$ according to the partition $\{\E_i\}$ of $s^{-1}(v)$. 
This gives a graph, $G_{os}$ say, with $G_{os}^0 = (G^{0} \setminus \{ v \}) \cup \{ v_1, v_2, \ldots, v_n \}$,
\[
	G_{os}^1 = \left( G^1 \setminus (r^{-1}(v) \cup s^{-1}(v)) \right) \bigcup \{ f_i \mid f \in r^{-1}(v), i =1,\ldots, n \} \bigcup \{ \bar{e}_i \mid e = e_i \},
\]
with range and source map that agree with those of $G$ when the edge is in $G^1$ but with
\[
	r_{os}(f_i) = v_i, \quad r_{os}(\bar{e}_i) = r_G(e_i), \quad s_{os}(f_i) = s_G(f), \quad  \text{ and } \quad s_{os}(\bar{e}_i) = v_i.
\]
Note that there are no edges between the $v_i$, that none of them support a loop of length one and that each of them emits exactly one edge.
We now use Lemma~\ref{easyCollapse} to collapse each of the $v_i$, this will yield the graph $E$. 
Thus,
\[
	G \Meq G_{os} \Meq E. 
\] 
\end{proof}

\begin{exam} \label{etExample}
As an application of Theorem~\ref{collapse} we see that 
\[
	\xymatrix{
		\star \ar@(dl,ul)^{4} \ar[r] & \bullet 
	}
	\Meq 
	\xymatrix{
		\circ \ar@/^1em/[r]^{2} &	\star \ar@/^1em/[l]^{2} \ar[r] & \bullet 
	}
	\Meq
	\xymatrix{
		\circ \ar@(dl,ul)^{4} \ar[r]^{2} & \bullet 
	}
\]
A number by an edge indicates that there are that many edges, so for instance in the leftmost graph there are four loops at $\star$.
This simplifies drawings, especially when dealing with infinite emitters.

The two leftmost graphs are seen to be move equivalent by collapsing the vertex $\circ$ in the middle graph.
The two rightmost are equivalent by collapsing $\star$ in the middle graph.
That the two outer graphs are $C^*$-equivalent was proved in \cite[Example 5.2]{eilersTomforde}. 
We return to this specific example again in the last section.
\end{exam}

\subsection{Move {\tt (T)}}

We will now discuss a move that pertains to infinite emitters. 
The idea is to make an infinite emitter emit infinitely to as many vertices as possible. 
We will prove that the move is in $\Meq$, which implies that it preserves stable equivalence of graph algebras.
It has been proved  in \cite{ers_amp} that it in fact preserves isomorphism of graph algebras. 

\begin{thm}[Move {\tt (T)}] \label{moveT}
Let $G = (G^0, G^1, r_G, s_G)$ be a graph and let $\alpha = \alpha_1 \alpha _2 \cdots \alpha_n$ be a path in $G$ and suppose that $A_G(s(\alpha_1), r(\alpha_1))| = \infty$. 
Let $E = (E^0, E^1, r_E, s_E)$ be the graph with vertex set $G^0$, edge set
 \[
	E^1 = G^1 \cup \{ \alpha^m \mid m \in \N \},
\]
and range and source maps that extend those of $G$, and have $r_{E}(\alpha^m) = r_{G}(\alpha)$ and $s_{E}(\alpha^m) = s_{G}(\alpha)$.
Then $G \sim_M E$.
\end{thm}
\begin{proof}
We will do the proof by induction on the length of $\alpha$.
If $|\alpha| = 1$ then since $A_G(s(\alpha_1), r(\alpha_1)| = \infty$, $G$ and $E$ are isomorphic. 

Suppose that $|\alpha| = 2$ and let $v = s(\alpha_2)$
If $\alpha_2$ is a loop of length one then $E \cong G$, so there is nothing to prove. 
We will therefore assume that $r(\alpha_2) \neq v$. 
If $\alpha_2$ is the only edge $v$ emits, we can collapse $v$ (as in Theorem~\ref{collapse}) in both $E$ and $G$, and get isomorphic graphs. 
So we will assume that $v$ emits at least two edges. 
Partition $s_G^{-1}(v)$ as $\E_1 = \{ \alpha_2 \}, \E_2 = s_G^{-1}(v) \setminus \{ \alpha_2 \}$.  
We can out-split $G$ at $v$ according to this partition and get a graph $G_{os}$ where $v$ is replaced by two vertices $v_1$ and $v_2$.
The first will emit only a copy of $\alpha_2$, whereas the second will emit copies of every other edge $v$ emitted. 
They will both receive copies of everything $v$ received. 
Arguing as above we can add infinitely many edges from $s(\alpha)$ to $r(\alpha)$ without changing move equivalence class. 
Doing so does not affect the edges going into or out of $v_1$ and  $v_2$. 
Hence we can out-amalgamate (i.e. use move {\tt (O)} backwards) them back to $v$. 
This yields a graph isomorphic to $E$, and since we only used our moves, we have $G \Meq E$.

Suppose now that $k > 2$ and that we have proved the theorem for paths of length less than $k$, and that $\alpha$ is a path of length $k$. 
Using the induction hypothesis on the path $\alpha_1 \alpha_2 \cdots \alpha_{k-1}$, we can, without changing move equivalence class, add infinitely many edges from $s(\alpha_1)$ to $s(\alpha_{k-1})$.
Hence there will be a path, $\beta$ of length two from $s(\alpha)$ to $r(\alpha)$ with $A_G(s(\beta_1),r(\beta_1)) = \infty$. 
Using that the induction hypothesis on $\beta$ we can add infinitely many edge from $s(\beta) = s(\alpha)$ to $r(\beta) = r(\alpha)$. 
Using the induction one final time, we remove the edges we added from $s(\alpha)$ to $s(\alpha_{k-1}$. 
Thus we have construct a graph isomorphic to $E$ using only allowed moves.
\end{proof}

For nice applications of move {\tt (T)} see \cite{ers_amp}. 
As a simple example to see how it works consider:
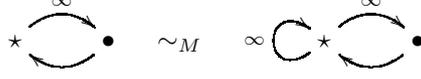
\[
	\xymatrix{
		\star \ar@/^1em/[r]^{\infty} & \bullet \ar@/^1em/[l] 
	}
	\quad
	\Meq
	\quad
	\xymatrix{
		\star \ar@(dl,ul)^{\infty} \ar@/^1em/[r]^{\infty} & \bullet \ar@/^1em/[l]
	}
\]
\linebreak 
The equivalence follows from Theorem~\ref{moveT}, since there are infinitely many edges from $\star$ to $\bullet$, and an edge from $\bullet$ to $\star$.
Using Theorem~\ref{collapse} to collapse $\bullet$, we see that both graphs are move-equivalent to 
\[
	\xymatrix{
		\star \ar@(dl,ul)^{\infty}
	}
\]
Hence, the algebras are all stably isomorphic to $\mathcal{O}_\infty$.

\section{Reductions (on graphs)}

The purpose of this section is to show the following proposition.

\begin{prop} \label{fruit}
Let $G$ be a graph with finitely many vertices such that $C^*(G)$ is purely infinite simple. 
There is some graph $E$ with finitely many vertices, no sinks, no sources, and satisfying the following:
\begin{enumerate}[(i)]
	\item Every vertex in $E$ supports a loop of length one,
	\item any two regular vertices in $E$ are connected by a path of regular vertices, 
	\item if $u \in E$ is an infinite emitter, then $u$ emits infinitely many edges to every vertex in $E$, and,
	\item $E \Meq G$. 
\end{enumerate}
\end{prop}

First we show that we do not have to worry about sinks and sources.

\begin{lem} \label{removeSources}
Let $G$ be a graph with finitely many vertices. 
If $C^*(G)$ is purely infinite simple then $G \Meq E$, where $E$ is some graph with no sinks or sources.
\end{lem}

For the following proof we need to know about the ideal structure of graph algebras. 
For this, and the notions of hereditary and saturated sets, we refer the reader to \cite{bhrs_idealstruct}.

\begin{proof}[Proof of Lemma~\ref{removeSources}.]
Since $C^*(G)$ is purely infinite simple every vertex in $G$ connects to a loop  \cite[Remark 2.16]{ddmt_arbgraph}. 
Hence $G$ has no sinks. 
Let $u \in G^0$ be a source and put $H = \{ v \in G^0 \mid u \geq v\} \setminus \{ u \}$. 
Because $u$ is a source $H$ is hereditary, and since $C^*(G)$ is simple the saturation of $H$ must be all of $G$. 
As no singular vertices are added when taking saturations, we must have that $u$ is regular. 
We can now use move {\tt (S)} to remove the sources from $G$. 
The resulting graph will have fewer vertices than $G$ although it may have more sources. 
Again we see that none of these new sources can be infinite emitters.
As there are only finitely many vertices in $G$ if we just keep removing sources we will eventually reach a graph with no sources. 
This graph, $E$ say, has no sinks and since we only used the move {\tt (S)} we must have $G \Meq E$. 
\end{proof}

We will now consider how the vertices interconnect.

\begin{lem} \label{boringPathStructure}
If $G$ is a graph with finitely many vertices, no sinks, no sources and $C^*(G)$ is simple, then any vertex in $G$ can reach any other vertex. 
\end{lem}
\begin{proof}
Since $G$ has no sources and only finitely many vertices every vertex can be reached by some loop.
To see this, fix some vertex $u_0 \in G^0$.
Either $u_0$ is on a loop or it is not. 
If it is we are done, so suppose it is not. 
As $u_0$ is not a source there is some vertex $u_1 \in G^0 \setminus \{ u_0 \}$ with an edge to $u_0$. 
Either $u_1$ is on a loop or it is not. 
If it is not then since $u_1$ is not a source there is some vertex $u_2 \in G^0 \setminus \{ u_0, u_1 \}$ that points to $u_1$. 
Continuing we must eventually reach a vertex that is on a loop since there are only finitely many vertices.
By \cite[Corollary 2.15]{ddmt_arbgraph} $G$ is cofinal because $C^*(G)$ is simple, in particular every vertex in $G$ can reach every loop in $G$. 
Let $u,v \in G^0$. 
There is some loop $\mu \in G^*$ such that $\mu \geq v$ and so
\[
	v \geq \mu \geq u. 
\]
Hence $v \geq u$.
\end{proof}

\begin{lem} \label{onlySimpleLoops}
If $G$ is a graph with finitely many vertices, no sinks, no sources and $C^*(G)$ is simple then $G \Meq E$ for some graph $E$ with no more vertices than $G$, no sinks, no sources and where every regular vertex supports a loop of length one.
\end{lem}
\begin{proof}
Let $v \in G^0$ be a regular vertex that does not support a loop of length one. 
Then we can collapse $v$ as in Theorem~\ref{collapse}. 
Doing this will yield a graph $F$ with one vertex less than $G$ and with $F \Meq G$. 
Repeating this process we will eventually get rid of all the regular vertices that do not support a loop of length one. 
\end{proof}

\begin{proof}[Proof of Proposition~\ref{fruit}.]
By Lemmas~\ref{removeSources} and \ref{onlySimpleLoops} we may assume that every regular vertex in $G$ supports a loop of length one, and that $G$ has neither sinks nor sources.
Let $u \in G^0$ be an infinite emitter and let $v \in G^0$ be such that $u$ emits infinitely many edges to $v$. 
Since $v \geq u$, by Lemma~\ref{boringPathStructure}, we can use Theorem~\ref{moveT} to add infinitely many edges from $u$ to itself. 
So without changing move-equivalence class, we can assume that every vertex in $G$ supports a loop of length one. 

Suppose we are given three distinct vertices $u,v,w$ such that $u$ is regular, $v$ is singular, there is an edge from $u$ to $v$ and one from $v$ to $w$. 
Fix a loop of length one based at $v$, $f$ say, and an edge from $v$ to $w$, $e$ say. 
Partition $s^{-1}(v)$ as $\mathcal{E}_1 = \{ f, e\}$ and $\mathcal{E}_2 = s^{-1}(v) \setminus  \{ e,f \}$.
Using move {\tt (O)} on this partition we replace $v$ by the regular vertex $v_1$ that supports a loop of length one and satisfies that there is an edge from $u$ to $v_1$ and one from $v_1$ to $w$, and the vertex $v_2$ which is just like $v$ was before except it does not emit $e$. 
The vertex $v_1$ can only reach the vertices $u$ can.
Hence, repeating this process a finite number of times will eventually lead to a graph were all the regular vertices are connected by paths of regular vertices. 

We have now shown that we can find a graph $F$ with finitely many vertices, no sinks or sources and satisfying $(i)$, $(ii)$ and $(iv)$. 
Let $u \in F^0$ be an infinite emitter and let $v$ be a vertex to which $u$ emits infinitely many edges. 
Let $w \in F^0$ be any vertex. 
By Lemma~\ref{boringPathStructure} $u \geq w$, so using move {\tt (T)} (Theorem~\ref{moveT}) we can add infinitely many edges from $u$ to $w$.
Doing the same for  every infinite emitter and every $w \in F^0$, we obtain a graph $E$ with finitely many vertices, no sinks or sources, and satisfying $(i)$, $(ii)$, $(iii)$ and $(iv)$.
\end{proof}

\section{Moves on matrices}

We now turn to the matrices. 
At the heart of Franks' classification of irreducible subshifts \cite{franks_flowequivalence} is simple matrix manipulation and since we wish to follow Franks, we will now consider how the moves we have listed can be used to manipulate adjacency matrices. 
Reordering the rows (and then the columns accordingly) in a matrix does not change the move-equivalence class, so we may assume that the rows which contain $\infty$ are at the bottom. 
The non-zero rows with no $\infty$ entries (i.e. those that correspond to a regular vertex) will be called ``regular'' rows, non ``regular'' rows, will be called ``singular'' rows. 
Similarly for the columns. 

While the matrix manipulations described in this section do not depend on the associated $C^*$-algebras being simple, it does depend on the matrices having a certain form.
So unlike our moves on graphs, the utility of these matrix manipulations in the non-simple case is probably limited. 
Nevertheless, we optimistically state the results in as much generality as possible.

Let $A$ be a square matrix and let $m$ be the number of ``regular'' rows in $A$.
Define an $n \times n$ matrix with $m$ regular rows by
\[
	J_{nm} = 
	\begin{pmatrix}
		1 	& 0 		& \cdots 	& 0 		& 0 		& \cdots 	& 0 \\
		0 	& 1 		& \cdots 	& 0 		& 0 		& \cdots 	& 0 \\
    		\vdots& \vdots 	& \ddots 	& \vdots 	& \vdots 	& \ddots 	& 0 \\
		0	& 0		& \cdots	& 1		& 0		& \cdots 	& 0 \\
		\infty	& \infty	& \cdots 	& \infty	& \infty	& \cdots 	& \infty \\
		\vdots& \vdots 	& \ddots 	& \vdots	& \vdots	& \ddots 	& \vdots \\
		\infty	& \infty	& \cdots 	& \infty	& \infty	& \cdots 	& \infty
	\end{pmatrix}.
\]

Proposition~\ref{fruit} shows that if $C^*(G_A)$ is purely infinite simple we can find an $n \times n$ matrix $B$ of the form 
\[
	B = 
	\begin{pmatrix}
		b_{11} 	& b_{12}	& \cdots 	& b_{1m}	& b_{1(m+1)} 	& \cdots 	& b_{1n} \\
		b_{12}	& b_{22}	& \cdots 	& b_{2m}	& b_{2(m+1)}	& \cdots 	& b_{2n} \\
    		\vdots 	& \vdots 	& \ddots 	& \vdots 	& \vdots 		& \ddots 	& \vdots \\
		b_{m1}	& b_{m2}	& \cdots	& b_{mm}	& b_{m(m+1)}	& \cdots 	& b_{mn} \\
		0		& 0		& \cdots 	& 0		& 0			& \cdots 	& 0 \\
		\vdots 	& \vdots 	& \ddots 	& \vdots	& \vdots		& \ddots 	& \vdots \\
		0	 	& 0		& \cdots 	& 0		& 0			& \cdots 	& 0
	\end{pmatrix},	
\]
such that 
\[
	A \Meq J_{nm} + B,
\]
and the top-left $m \times m$ corner of $J_{nm} + B$, as well as all of it, is irreducible. 

We will now show that we can manipulate matrices in way similar to \cite[Corollary 2.2]{franks_flowequivalence}.
Due to the asymmetry between rows and columns we will need to do two proofs. 
This was not an issue for Franks, as $A$ is flow equivalent to $B$ if and only if $A^T$ is flow equivalent to $B^T$.
For us, there is a big difference between rows and columns, as seen for instance in the difference between out- and in-splittings, or the fact that we can remove (regular) sources, but not sinks. 

We first handle rows. 

\begin{lem} \label{rowMove}
Let $A$ be a square matrix with no zeros on the diagonal. 
Let $i,j$ be two distinct indices. 
If $a_{ij} \neq 0$ and row $j$ contains no $\infty$, then 
\[
	A \Meq \begin{pmatrix}	
			a_{11} 		& a_{12} 		& \cdots 	& a_{1j} 			& \cdots 	& a_{1n} \\
			\vdots 		& \vdots 		& 		& \vdots 			&            	& \vdots \\
			a_{i1} + a_{j1}	& a_{i2} + a_{j2}	&		& a_{ij} + a_{jj} - 1		&		& a_{in} + a_{jn} \\
			\vdots 		& \vdots 		& 		& \vdots 			&            	& \vdots \\
			a_{n1} 		& a_{n2} 		& \cdots 	& a_{nj} 			& \cdots 	& a_{nn} \\
		\end{pmatrix}.
\]
The matrix on the right is $A$ with the $j$'th row added to the $i$'th row but where we subtract $1$ at the $(i,j)$'th entry.
Note that this matrix has no zeros on the diagonal, and if the top-left corner was irreducible in $A$ then it is still irreducible. 
The same is true for all of $A$ and all of the right-hand side.
\end{lem}
\begin{proof}
In terms of graphs, what we are doing is choosing two vertices in the graph $G_A$, $u$ and $v$ say, with an edge, $f$ say, from $u$ to $v$. 
A new graph, $E$, is then formed by removing $f$ but adding for each edge $e \in s^{-1}(v)$ an edge $\bar{e}$ with $s(\bar{e}) = u$ and $r(\bar{e}) = r(e)$. 
We claim that $E \Meq G_A$ if $v$ is not an infinite emitter. 

Since $v$ is not singular, we can use move {\tt (I)} at it. 
Partition $r^{-1}(v)$ as $\mathcal{E}_1 = \{ f \}$ and $\mathcal{E}_2 = r^{-1}(v) \setminus \{ f \}$. 
As there is a loop of length one based at $v$. $\mathcal{E}_2$ is not empty. 
Insplitting according to this partition replaces $v$ with two new vertices, $v_1$ and $v_2$. 
The vertex $v_1$ only receives one edge, and that edge comes from $u$, the vertex $v_2$ receives the edges $v$ received except $f$ and also receives one edge from $v_1$ for each loop of length one based at $v$. 
Both vertices emit copies of the edges $v$ emitted, and do so in such a way that there is no loop of length one based at $v_1$: instead it emits one edge to $v_2$ for each loop of length one based at $v$. 
Collapsing $v_1$ as in Theorem~\ref{collapse} yields $E$. 
\end{proof} 

Above we proved move-equivalence of two matrices and hence we have that their algebras are stably isomorphic. 
It follows from \cite[Corollary 2.5]{aalp_gauge} that if the associated graphs are row-finite then the algebras are isomorphic. 
We now turn to columns. 

\begin{lem} \label{columnMove}
Let $A$ be a matrix with no zeros on the diagonal. 
Let $i,j$ be two distinct indices. 
If $a_{ji} \neq 0$, then 
\[
	A \Meq \begin{pmatrix}	
			a_{11}	& \cdots 	& a_{1i} + a_{1j}		& \cdots 	& a_{1n} \\
			a_{21}	& \cdots 	& a_{2i} + a_{2j}		& \cdots 	& a_{1n} \\
			\vdots		&  		& \vdots 			&            	& \vdots \\
			a_{j1}		& \cdots 	& a_{ji} + a_{jj} - 1		& \cdots	& a_{jn} \\
			\vdots 	& 		& \vdots			& 	       	& \vdots \\
			a_{n1}	& \cdots	& a_{ni} + a_{nj}		& \cdots 	& a_{nn} \\
		\end{pmatrix}.
\]
The matrix on the right is $A$ with the $j$'th column added to the $i$'th column but where we subtract $1$ at the $(j,i)$'th entry.
Note that the matrix on the right has no zeros on the diagonal, and that if the top-left corner was irreducible in $A$ then it is still irreducible on the right-hand side.  
The same holds for all of $A$ and all of the right-hand side.
\end{lem}
\begin{proof}
In terms of graphs, what we are doing is choosing two vertices in the graph $G_A$, $u$ and $v$ say, with an edge, $f$ say, from $v$ to $u$. 
A new graph, $E$, is then formed by removing $f$ but adding for each edge $e \in r^{-1}(v)$ an edge $\bar{e}$ with $s(\bar{e}) = s(e)$ and $r(\bar{e}) = u$. 
We claim that $E \Meq G_A$. 

Partition $s^{-1}(v)$ as $\mathcal{E}_1 = \{ f \}$ and $\mathcal{E}_2 = s^{-1}(v) \setminus \{ f \}$. 
Since there is a loop of length one based at $v$, $\mathcal{E}_2$ is not empty, so we can use move {\tt (O)}. 
Doing so yields a graph just as $G$ but where $v$ is replaced by two vertices, $v_1$ and $v_2$.
The vertex $v_1$ receives a copy of everything $v$ did including an edge from $v_2$ for each loop of length one based at $v$, it emits only one edge, and that edge has range $u$. 
The vertex $v_2$ also receives a copy of everything $v$ did, and it emits everything $v$ did, except $f$. 
Since $v_1$ is regular and not the base of a loop of length one, we can collapse it (Theorem~\ref{collapse}) and thereby obtain $E$.
\end{proof}

Following Franks, we use the above to do matrix operations (see \cite[Theorem 2.4]{franks_flowequivalence}). 

\begin{prop} \label{matrixOperations}
Suppose we are given an irreducible matrix $J_{nm} + B$ where the top-left $m \times m$ corner is irreducible and the last $n-m$ rows of $B$ are zero. 
If we form $B'$ from $B$ by adding any column to any other column, or by adding a non-zero row to any other non-zero row, then $J_{nm} + B \Meq J_{nm} + B'$.
\end{prop}

Clearly adding a zero row to any other row changes nothing, and neither does adding any row to one of the zero rows since we add the matrix $J_{nm}$ to $B$ when we form $A$, so every entry of the last $n-m$ rows will be $\infty$. 
We have added the requirement, since there is very little gained by not doing so, and, more importantly, this is the way we intend to use it.  

\begin{proof}[Proof of Proposition~\ref{matrixOperations}.]
First notice that we can add a ``singular'' column to any other column by Lemma~\ref{columnMove}.

Let $A = J_{nm} + B$ and suppose we are given two distinct indices $i,j$ such that $j \leq m$.
Let $B'$ be the matrix obtained by adding the $j$'th column of $B$ to the $i$'th.
If $a_{ji}$ is non-zero then $J_{nm} + B \Meq J_{nm} + B'$ by Lemma~\ref{columnMove}.
Hence applying  Lemma~\ref{columnMove} shows that for any two distinct indices $l,k$ we can subtract row $l$ in $B$ from row $k$ in $B$ without changing move-equivalence class, provided that the resulting matrix is non-negative and $b_{kl} > 0$ after the subtraction.

Any two vertices in $G_A$ are connected, so we can find a sequence of distinct indices $j = i_0, i_1, \ldots, i_k = i$ such that the entries $a_{i_{l}i_{l+1}}$ all are non-zero. 
Since $a_{ji_{1}}$ is non-zero we can add column $i_{1}$ to column $j$ (in $B$). 
The new column $j$ will be non-zero at entry $a_{ji_{2}}$, so we can add column $i_{2}$ to the new column $j$. 
Carrying on, we will eventually add column $i$ to column $j$.
The version of column $j$ we have now is a sum of the columns $j = i_0, i_1, \ldots, i_k = i$ from $B$. 
We can now subtract column $i_{k-1}$ from column $j$, then we subtract column $i_{k-2}$ and so on. 
In the end, we will have that we have only added column $i$ to column $j$, as we wanted. 

The proof for rows is very similar. 
The only change, is that we need to use that the regular vertices in $G_A$ are connected, so we only work with non-zero rows of $B$. 
But that is the case, since the top-left $m \times m$ corner of $A$ is irreducible. 
\end{proof}

\section{The standard form}

We will now, following Franks, put our matrices in a standard form.
There are two cases, one where the matrices have ``singular'' rows, and one where they do not. 

\subsection{Graphs with at least one singularity}

In this subsection we consider the case where we are given a matrix $A$ of the form $A = J_{nm} + B$ with $m \neq n$ and the last $n-m$ rows of $B$ filled with zeros. 
We copy the approach Franks takes in the last part of section $2$ and all of section $3$ in \cite{franks_flowequivalence}. 
Like Franks does, we first show that we can enlarge $B$ without changing move-equivalence class.
Moreover, we can, unlike Franks, get the determinant of the top-left regular corner of the enlarged matrix to be $0$, this is the reason we will not need the Cuntz splice later.
The key point is that we can adjoin almost any row to $A$ without changing the $\Meq$-class. 

\begin{lem} \label{adjoinRow}
Suppose we are given an irreducible matrix $A = J_{nm} + B$ where the top-left $m \times m$ corner is irreducible,  $m \neq n$ and the last $n - m$ rows of $B$ are filled with zeros. 
Let $x_1, x_2, \ldots,x_n$ be non-negative integers that are not all zero and let $j$ with $m < j \leq n$ be given.
If
\[
	A' = 	\begin{pmatrix}
			x_j		& x_1		& \cdots 	& x_{j-1}	& x_j		& x_{j+1}	& \cdots & x_n \\
			a_{1j}		& a_{11} 	& \cdots 	& a_{1(j-1)}	& a_{1j}	& a_{1(j+1)}	& \cdots & a_{1n} \\
			\vdots 	& \vdots 	& \ddots 	& \vdots 	& \vdots 	& \vdots 	& \ddots & \vdots \\
			a_{mj}		& a_{m1}	& \cdots 	& a_{m(j-1)}	& a_{mj}	& a_{m(j+1)}	& \cdots & a_{mn} \\	
			\infty 		& \infty	& \cdots 	& \infty	& \infty	& \infty 	& \cdots & \infty \\
			\vdots 	& \vdots 	& \ddots 	& \vdots 	& \vdots 	& \vdots 	& \ddots & \vdots \\
			\infty 		& \infty	& \cdots 	& \infty	& \infty	& \infty 	& \cdots & \infty 
		\end{pmatrix},
\]
then $A \Meq A'$. 
\end{lem}
\begin{proof}
Let $u_1, u_2, \ldots u_n$ be the vertices of $G_A$, the labeling chosen so that the $i$'th row and column of $A$ describes the edges going out of and into, respectively, $u_i$. 
Since $j > m$, $u_j$ is an infinite emitter, and as $A = J_{mn} + B$, $u_j$ emits infinitely many edges to every vertex in $G_A$. 
Partition $s^{-1}(u_j)$ into two sets, one containing $x_1$ edges to $u_1$, $x_2$ edges to $u_2$ and so on, the other containing everything else. 
Out-splitting according to this partition produces a graph with adjacency matrix $A'$.
Hence $A' \Meq A$. 
\end{proof}

\begin{prop} \label{asBigAsYouLike}
Suppose we are given an irreducible matrix $A = J_{nm} + B$ where the top- left $m \times m$ corner is irreducible, $m \neq n$ and the last $n - m$ rows of $B$ are filled with zeros. 
For any $k \in \N$ such that $k \geq n+2$ we can find a $k \times k$ matrix $C$, satisfying the following:
\begin{enumerate}[(i)]
	\item All the entries in the first $k-(n-m)$ rows of $C$ are non-negative, 
	\item at least one of the entries in $C$ is 1,
	\item all the entries in the last $n-m$ rows of $C$ are zero, and,
	\item $J_{nm} + B \Meq J_{k(k - (m-n))} + C$.
\end{enumerate}
\end{prop}
\begin{proof}
If $m = 0$, we can use Lemma~\ref{adjoinRow} to replace $A$ with an irreducible $(n+1) \times (n+1)$ matrix with one ``regular'' row, and the regular top-left $1 \times 1$ corner irreducible, i.e. non-zero. 

Suppose now that $m > 0$ and that $k \geq n+2$ is given. 
Since $A$ is irreducible, there is at least one non-zero entry of $B$. 
Using that entry and row and column additions (Proposition~\ref{matrixOperations}), we can find a matrix $A' = J_{nm} + B'$, such that $A' \Meq A$ and every entry in first $m$ rows of $B$ are non-zero and every entry in the last $n-m$ rows are $0$.
We can now use Lemma~\ref{adjoinRow} to add rows of the form $(2,1,1,1\ldots,1,2,1,1,\ldots,1)$, where the second $2$ is at the $(m+1)$'st entry, to $A'$ until we have a $k \times k$ matrix $A''$. 
Since $k \geq n+2$ we add at least one such row (in the case where the original matrix had $m \neq 0$, we actually add at least two such rows). 
We have that $A'' \Meq A$, that $A'' = J_{k(k-(m-n))} + C$ for some matrix $C$ with all the entries in the first $k-(n-m)$ rows non-zero, all the entries in the last $n-m$ rows zero, and one row that look like $(1,1,\ldots, 1,1,2,1,1,\ldots, 1)$, with the $2$ in a singular column. 
In particular $C$ contains a $1$. 
\end{proof}

Next we show that we can get a column of $1$'s. 

\begin{prop} \label{aColumnOfOnes}
Suppose we are given a matrix $A = J_{nm} + C$ where $m \neq n$. 
Suppose all the entries in the first $m$ rows of $C$ are non-zero, all the entries in the last $n-m$ rows are zero, and $C$ has $1$ in at least one entry. 
There exists an $n + 2 \times n + 2$ matrix $D$ such that:
\begin{enumerate}[(i)]
	\item The last $(n+2)-(m+2)$ rows of $D$ are identically zero, 
	\item any entry in any other row is strictly positive,
	\item $d_{i1} = 1$ for all $1 \leq i \leq m+2$,
	\item the determinant of the top-left $m+2 \times m+2$ corner of $D$ is zero, and, 
	\item $J_{nm} + C \Meq J_{(n+1)(m+1)} + D$.
\end{enumerate} 
\end{prop}
\begin{proof}
We can use the algorithm described by Franks in \cite[Proposition 2.9]{franks_flowequivalence} to obtain an $n \times n$ matrix $C'$ with the last $n-m$ rows identically zero, any entry in any other row non-zero, a column consisting only of $1$'s in the first $n$ places, and such that $J_{nm} + C \sim J_{nm} + C'$.
If the column of $1$'s described the edges going into a regular vertex, we would just be a permutation away from having it be the first column.
To ensure that this happens we will use Lemma~\ref{adjoinRow}. 
Let $j_0$ be the index of the column consisting only of $1$'s and let $j = \max\{j_0, m+1\}$. 
Then $j > m$ so we can use Lemma~\ref{adjoinRow} to add the row $(2,1,1,\ldots,1,2,1,\ldots,1)$ to $A' = J_{nm} + C'$, thereby getting a matrix $A'' = J_{(n+1)(m+1)} + C''$.
If $j = j_0$ then we now have that the first column of $C''$ consist entirely of $1$'s. 
If $j = m+1$ then some regular column consists entirely of $1$'s, since the added $2$ in $C''$ will be in a singular column. 
So after a permutation the first column of $C''$ consists entirely of $1$'s. 
Note that $C''$ now satisfies $(i), (ii), (iii)$ and $(v)$, and that the top row of $C''$ is $(1,1,\ldots, 1,2,1,\ldots, 1)$ with the $2$ in a singular column. 
Using Lemma~\ref{adjoinRow} we can add row of the form $(3,1,1,\ldots,1,3,1,\ldots,1)$ to $A''$ resulting in $A''' = J_{(n+2)(m+2)} + D'$, where $D'$ satisfies $(i), (ii),$ and $(v)$, and the regular part of the two top rows of $D'$ are $(2,1,\ldots, 1)$ and the second column of $D$ consists of $1$s.
Hence the top-left $(m+2) \times (m+2)$ corner of $D'$ has two identical rows, and therefore has determinant zero. 
Preforming a permutation on $D'$ that interchanges the second and first rows, we get a matrix $D$ that satisfies $(i), (ii), (iii), (iv),$ and $(v)$. 
\end{proof}

We can now put the matrices in a canonical form.
We will argue by induction on the rank of the matrix. 

\begin{lem} \label{basisRankOne}
Suppose we are given an $n \times n$ matrix $C$ where all the entries in the first $m$ rows are non-zero, all the entries in the last $n-m$ rows are zero, and the first column is $(d,d,\ldots, d, 0, 0, \ldots, 0)^T$, where $d$ is the $\gcd$ of the non-zero entries of $C$.
If $C$ is rank $1$, then 
\[
	J_{mn} + C \sim J_{mn} + D,
\]
where $D$ is the $n \times n$ matrix where all the entries in the first $m$ rows are $d$, and all the entries in the last $n-m$ rows are zero.
\end{lem}
\begin{proof}
Since $C$ is rank $1$ and integer valued, each column is a multiple of $(d,d,\ldots, d,0, 0, \ldots, 0)^T$. 
We can use column subtraction (Proposition~\ref{matrixOperations}) to form $D$. 
\end{proof}

\begin{thm} \label{standardForm}
Suppose we are given a matrix $A = J_{nm} + C$ where $m \neq n$. 
Suppose all the entries in the first $m$ rows of $C$ are non-zero, all the entries in the last $n-m$ rows are zero, and the first column is $(d,d,\ldots, d, 0, 0, \ldots, 0)^T$, where $d$ is the $\gcd$ of the non-zero entries of $C$.
Let $d_1, d_2, \ldots, d_k$ denote the non-zero elementary divisors of the first $m$ rows of $C$ order so that $d_i$ is a factor of $d_{i+1}$. 
Define 
\[
	B =
	\begin{pmatrix}
		0 		& 0 		& 0 		& \cdots 	& 0 & d_k & d_k & \cdots& d_k \\
		d_1 		& 0 		& 0 		& \cdots 	& 0 & 0 & 0 & \cdots & 0  \\
		0 		& d_2		& 0 		& \cdots 	& 0 & 0 & 0 & \cdots & 0 \\
		0		& 0		& \ddots 	& 		& \vdots & \vdots & \vdots & & \vdots \\
		\vdots 	& \vdots	&		& \ddots 	& \vdots & \vdots & \vdots & & \vdots \\
		0		& 0		& 0		& \cdots	& d_{k-1} & 0 & 0 & \cdots & 0 \\
		\hline
		0		& 0		& 0		& \cdots 	& 0	& d_k & d_k & \cdots & d_k \\				 
		\vdots 	& \vdots 	& \vdots 	& \ddots 	& \vdots & \vdots 	& \vdots & \ddots & \vdots \\
		0		& 0		& 0		& \cdots 	& 0	& d_k & d_k & \cdots & d_k \\
		\hline
0		& 0		& 0		& \cdots 	& 0	& 0 & 0 & \cdots & 0 \\				 
		\vdots 	& \vdots 	& \vdots 	& \ddots 	& \vdots & \vdots 	& \vdots & \ddots & \vdots \\
		0		& 0		& 0		& \cdots 	& 0	& 0 & 0 & \cdots & 0 
	\end{pmatrix}.
\]
There are $m$ rows above the second line and $n-m$ below it.
 
If the ``regular'' top-left corner of $A = J_{nm} + C$ has determinant $0$, then $A \Meq J_{nm} + B$.
\end{thm}
\begin{proof}
We mimic Franks' techniques and do the proof by induction on the rank of $C$.
The case of rank $1$ is Lemma~\ref{basisRankOne}.

Let us assume that the rank of $C$ is at least $2$. 
Since $\det(C) = 0$, $C$ has at least three rows.  
Doing exactly what Franks does in the proof of \cite[Proposition 3.1]{franks_flowequivalence}, we get a matrix $B'$ of the form 
\[
	B' = \begin{pmatrix}
		0	& * & \cdots & * \\
		d_1 	& 0 & \cdots & 0 \\
		0	& * & \cdots & * \\
		\vdots& \vdots & \ddots & \vdots \\
		0	& * & \cdots & * 
	\end{pmatrix},
\]
such that $J_{nm} + B \Meq J_{nm} + B'$. 
The matrix of $*$'s has smaller rank than $B'$, the determinant of its top-left ``regular'' corner  is $0$, and the first row will be $(d', d', \ldots, d', 0, 0, \ldots, 0)$, where $d'$ is $\gcd$ of the remaining non-zero entries. 
So, by induction, we can put it in the desired form. 
\end{proof}

\subsection{Regular graphs}

In this subsection we will consider graphs (and corresponding matrices) with no infinite emitters. 
That is, matrices of the form $A = J_{nn} + B = I + B$ for some matrix $B$. 
Since there are no ``singular'' rows (or columns) Proposition~\ref{matrixOperations} gives us the following. 

\begin{thm}
Suppose $B$ is a non-negative square integer matrix and $A = I + B$ is irreducible. 
If $B'$ is obtained from $B$ by adding any row to a different row or adding any column to a different column then $A \Meq I + B'$.
\end{thm}

This is parallel to \cite[Theorem 2.4]{franks_flowequivalence}.
Since the proofs Franks gives really only depend on matrix operations, we get the following two results (\cite[Corollary 2.6 and Theorem 3.3]{franks_flowequivalence}) by copying proofs. 

\begin{thm} \label{regularLarge}
Let $A$ be a non-negative integer matrix such that $C^*(G_A)$ is purely infinite simple. 
There exists an $N \in \N$ such that for all $n > N$, there is a strictly positive $n \times n$ integer matrix $B$ with $A \Meq I + B$, and $\det(I - A) = \det(-B)$
\end{thm}
\begin{proof}
We copy Franks' proof to get $A \Meq I + B$. 
The determinant condition follows since $A$ and $I + B$ not only are move-equivalent they are also flow equivalent. 
\end{proof}

\begin{thm} \label{regularStandard}
Suppose that $B$ is an $n \times n$, $n > 1$, strictly positive matrix with elementary divisors $d_1, d_2, \ldots, d_n$, each $d_i$ a factor of $d_{i+1}$. 
Let $A = I + B$ and let $k = \rank(-B)$. 
If $\det(-B) \leq 0$ then $A \Meq I + B'$ where  
\[
	B' =
	\begin{pmatrix}
		0 		& 0 		& 0 		& \cdots 	& 0 & d_k & d_k & \cdots& d_k \\
		d_1 		& 0 		& 0 		& \cdots 	& 0 & 0 & 0 & \cdots & 0  \\
		0 		& d_2		& 0 		& \cdots 	& 0 & 0 & 0 & \cdots & 0 \\
		0		& 0		& \ddots 	& 		& \vdots & \vdots & \vdots & & \vdots \\
		\vdots 	& \vdots	&		& \ddots 	& \vdots & \vdots & \vdots & & \vdots \\
		0		& 0		& 0		& \cdots	& d_{k-1} & 0 & 0 & \cdots & 0 \\
		0		& 0		& 0		& \cdots 	& 0	& d_k & d_k & \cdots & d_k \\				 
		\vdots 	& \vdots 	& \vdots 	& \ddots 	& \vdots & \vdots 	& \vdots & \ddots & \vdots \\
		0		& 0		& 0		& \cdots 	& 0	& d_k & d_k & \cdots & d_k
	\end{pmatrix}
\]
\end{thm}

\section{Geometric classification} \label{sec:classification}

We can now prove our main theorem. 
We begin by giving short proofs of two lemmas that are certainly well known though the author has been unable to find a good reference.

\begin{lem} \label{piOrAF}
Suppose $G$ and $E$ are finite graphs.
If $G \Keq E$ and $C^*(G)$ (and hence $C^*(E)$) is simple, then $C^*(G)$ and $C^*(E)$ are either both AF or both purely infinite
\end{lem}
\begin{proof}
A simple graph algebra is either purely infinite or AF \cite[Remark 2.16]{ddmt_arbgraph}. 
The positive cone of the $K_0$-group will tell us which case we are in, if it is all of $K_0$, then the algebras are purely infinite, if it is not, then they are AF.
\end{proof}

\begin{lem} \label{detectSingularities}
Suppose $G$ and $E$ are finite graphs.
If $G \Keq E$ then $G$ and $E$ have the same number of singularities. 
\end{lem}
\begin{proof}
The difference between the rank of the free abelian part of the $K_0$-group and the rank of the $K_1$-group is the number of singularities. 
See \cite[Theorem 3.1]{ddmt_kthygraph}.
\end{proof}

We now can now prove our main theorem.

\begin{proof}[Proof of Theorem~\ref{mainThm}.]
We clearly have $G \Ceq E \implies C \Keq E$ and $G \Meq E \implies G \Ceq E$, and in the purely infinite simple case we also have $G \MCeq E \implies G \Ceq E$.
Hence, the non-trivial part of the theorem is to show that only if the algebra is purely infinite simple can it have no singularities, and the implications $G \Keq E \implies G \Meq E$ when $G$ has at least one singularity and $G \Keq E \implies G \MCeq E$ when $G$ has no singularities. 
Suppose we are given two graphs $G$ and $E$ with $G \Keq E$.

By Lemma~\ref{piOrAF} either both $C^*(G)$ and $C^*(E)$ are purely infinite or they are both AF. 
We will deal with the AF case first. 
Since $C^*(G)$ is AF $G$ has no loops (\cite[Remark 2.16]{ddmt_arbgraph}) and since $C^*(G)$ is unital $G$ has finitely many vertices. 
Hence, $G$ has a sink. 
Let $v \in G^0$ be a sink and let $H = \{ v \}$. 
Since $H$ is hereditary and $C^*(G)$ is simple all vertices in $G$ must be in the saturation of $H$, so $G$ has precisely one sink and no infinite emitters.
The same argument shows that $E$ has exactly one sink and no infinite emitters. 
Using again that $G$ and $E$ have finitely many vertices we see by repeated applications of move {\tt (S)} that 
\[
	G \Meq \bullet \Meq E. 
\]
 
Suppose now that $C^*(G)$ (and so $C^*(E)$) is purely infinite.
As before both $G$ and $E$ have finitely many vertices.
By Lemma~\ref{detectSingularities} $G$ and $E$ have the same number of singularities and since the graph algebras are purely infinite simple, the graphs have no sinks (see \cite[Remark 2.16]{ddmt_arbgraph}). 
Hence $G$ and $E$ have the same number, $k$ say, of infinite emitters.

Suppose $G$ has at least one infinite emitter, i.e. $k \geq 1$.
By Propositions~\ref{asBigAsYouLike} and~\ref{aColumnOfOnes} we can find matrices $C$ and $D$ of the same size, $n$ say, such that:
\begin{enumerate}[(i)]
	\item all the entries in the first $n-k$ rows of $C$ and $D$ are non-zero,
	\item all the entries in the last $k$ rows of $C$ and $D$ are zero, 
	\item the first column of both $C$ and $D$ is $(1,1,\cdots, 1,0, 0 \cdots, 0)^T$,
	\item the determinant of the top-left ``regular'' corner is zero,
	\item $J_{n(n-k)} + C \Meq A_G$, and,
	\item $J_{n(n-k)} + D \Meq A_E$.	
\end{enumerate}
From \cite[Theorem 3.1]{ddmt_kthygraph} we get that
\begin{align*}
	\coker C^T 	& \cong \coker \begin{pmatrix} c_{11} & \cdots & c_{1n} \\ \vdots & \ddots & \vdots \\ c_{(n-k)1} & \cdots & c_{(n-k)n}  \end{pmatrix}^T \cong K_0(C^*(G)) \\
			& \cong K_0(C^*(E)) \cong \begin{pmatrix} d_{11} & \cdots & d_{1n} \\ \vdots & \ddots & \vdots \\ d_{(n-k)1} & \cdots & d_{(n-k)n}  \end{pmatrix}^T \cong \coker D^T. 
\end{align*}
Hence, $C$ and $D$ have the same elementary divisors. 
By Theorem~\ref{standardForm} we then have $J_{n(n-k)} + C \Meq J_{n(n-k)} + D$, so
\[
	A_G \Meq J_{n(n-k)} + C \Meq J_{n(n-k)} + D \Meq A_E. 
\]
Thus, $G \Meq E$.

We now consider the case where $k = 0$.
Using the Cuntz splice we can find graphs $\widetilde{G}$ and $\widetilde{E}$ such that $\det(I - A_{\widetilde{G}})$ and $\det(I - A_{\widetilde{E}})$ both are non-positive, and $G \MCeq \widetilde{G}$ and $E \MCeq \widetilde{E}$. 
Now we can use Theorem~\ref{regularLarge} to find matrices $C,D$ of the same size such that $C \Meq A_{\widetilde{G}}$, $D \Meq A_{\widetilde{E}}$ and $\det(-C)$ and $\det(-D)$ both are non-positive. 
The $K$-theory argument from the previous case again works to show that $C$ and $D$ have the same elementary divisors. 
Hence, by Theorem~\ref{regularStandard}
\[
	A_{\widetilde{G}} \Meq I + C \Meq I + D \Meq A_{\widetilde{E}}.
\] 
Therefore,
\[
	G \MCeq \widetilde{G} \Meq \widetilde{E} \MCeq E. 
\]
\end{proof}

\begin{rema}
It is interesting to note that in the presence of an infinite emitter we do not need the Cuntz splice to fix the sign of the determinant. 
This suggests, at least to the author, that even in the non-simple case we should be able to do something like a Cuntz splice when we have infinite emitters. 
We most likely need to assume that the infinite emitter and the vertex we wish to Cuntz splice at interconnect in some way, as this will not always be the case when the graph algebra is not simple. 
\end{rema}

\begin{rema} \label{etRemark}
In example~\ref{etExample} we saw that 
\[
	\xymatrix{
		\star \ar@(dl,ul)^{4} \ar[r] & \bullet 
	}
	\Meq
	\xymatrix{
		\circ \ar@(dl,ul)^{4} \ar[r]^{2} & \bullet 
	}
\]
In \cite[Example 5.2]{eilersTomforde} these graphs are studied as examples of graphs whose algebras have exactly one ideal, and where it would be hard to show stable isomorphism of the associated algebras without using $K$-theoretic classification. 
Since we were able to do this, one might hope that we can extend Theorem~\ref{mainThm} to the one-ideal, or even general non-simple, case. 
However, in the interest of full disclosure we should note that the results in \cite{eilersTomforde} can be used to classify all graphs of the form 
\[
	\xymatrix{
		\circ \ar@(dl,ul)^{n} \ar[r]^{k} & \bullet 
	},
	\quad 
	n \geq 2, k \geq 1.
\]
And many others. 
That we could produce a move-equivalence between the two graphs considered in \cite[Example 5.2]{eilersTomforde} hinges on the fact that $2$ divides $4$. 
So even though 
\[
	\xymatrix{
		\circ \ar@(dl,ul)^{5} \ar[r]^{1} & \bullet 
	}
	\Ceq
	\xymatrix{
		\circ \ar@(dl,ul)^{5} \ar[r]^{3} & \bullet 
	}	
\]
by \cite[Theorem 5.1]{eilersTomforde}, it is unclear (to the author) if the graphs are move-equivalent.
\end{rema}

We can say something in the non-simple case, for instance we have the following.

\begin{prop} \label{AFsing}
If $C^*(G)$ is a unital AF algebra then there is a graph $E$ with only singular vertices such that $G \Meq E$.
\end{prop}
\begin{proof}
Since $C^*(G)$ is unital $G$ only has finitely many vertices. 
We can, in a finite number of steps, collapse (using theorem~\ref{collapse}) all the regular vertices of $G$.
\end{proof}

In \cite{ers_amp} it is shown that move {\tt (T)} can be used to show that any two singular graphs with no breaking vertices which are $K$-equivalent actually are move-equivalent. 
This will not always be useful here, as we expect the graph $E$ from Proposition~\ref{AFsing} often will have breaking vertices.
Seen together with Remark~\ref{etRemark} this suggests that we might need more moves to handle the non-simple case.

\section*{acknowledgement}
The author is grateful for many fruitful conversations about the topic of this paper with his advisor S{\o}ren Eilers.

\providecommand{\bysame}{\leavevmode\hbox to3em{\hrulefill}\thinspace}
\providecommand{\MR}{\relax\ifhmode\unskip\space\fi MR }
\providecommand{\MRhref}[2]{%
  \href{http://www.ams.org/mathscinet-getitem?mr=#1}{#2}
}
\providecommand{\href}[2]{#2}


\end{document}